\documentclass[10pt,reqno]{amsart}
\usepackage{amsmath}
\usepackage{amssymb}
\usepackage{amsthm}
\usepackage{eepic,epic}
\usepackage{epsfig}
\usepackage{graphicx}

\newlength{\defbaselineskip}
\newcommand{\setlinespacing}[1]%
           {\setlength{\baselineskip}{#1 \defbaselineskip}}

\numberwithin{equation}{section}

\newtheorem{thm}{Theorem}[section]

\newtheorem{lem}[thm]{Lemma}
\newtheorem{prop}[thm]{Proposition}

\theoremstyle{definition}

\theoremstyle{remark}
\newtheorem{rem}[thm]{Remark}
\numberwithin{equation}{section}

\begin{document}

\title[Morawetz estimates]
{On Morawetz estimates for the elastic wave equation}

\author{Seongyeon Kim and Ihyeok Seo}

\thanks{  }


\address{Department of Mathematics Education, Jeonju University, Jeonju 55069, Republic of Korea}
\email{sy\_kim@jj.ac.kr}

\address{Department of Mathematics, Sungkyunkwan University, Suwon 16419, Republic of Korea}
\email{ihseo@skku.edu}

\begin{abstract}
We establish Morawetz-type estimates for solutions to the elastic wave equation with singular weights of the form $|x|^{-\alpha}$ or $|(x,t)|^{-\alpha}$. 
In particular, we show that space-time weights $|(x,t)|^{-\alpha}$ admit stronger singularities and require weaker regularity assumptions on the initial data compared to purely spatial weights $|x|^{-\alpha}$.
\end{abstract}

\maketitle

\section{Introduction}\label{sec1}
In this paper, we consider the Cauchy problem for the elastic wave equation:
\begin{equation} \label{eq}
\begin{cases}
\partial_t^2 u -\mu\Delta u -(\lambda + \mu)\nabla \mathrm{div}u = 0, \\
u(x,0)=f(x), \quad \partial_t u(x,0)=g(x),
\end{cases}
\end{equation}
where $f, g\colon \mathbb{R}^n \rightarrow \mathbb{R}^n$ and $u\colon \mathbb{R}^n\times \mathbb{R} \rightarrow \mathbb{R}^n$. 
The Lam\'e coefficients $\lambda,\mu\in\mathbb{R}$ are assumed to satisfy the ellipticity condition
\begin{equation} \label{ellipticity}
\mu > 0, \quad \lambda + 2\mu > 0.
\end{equation}
The equation is widely used to model wave propagation in elastic media, where $u$ represents the displacement field of the material (see, e.g., \cite{LL,MH}).

The condition \eqref{ellipticity} ensures the ellipticity of the Lam\'e operator $\Delta^*$, which is defined by
\begin{equation*}
\Delta^*u =\mu\Delta u +(\lambda + \mu)\nabla \mathrm{div}u.
\end{equation*}
Indeed, it acts naturally on vector fields via the \emph{Helmholtz decomposition}, which states that any $f\in [L^2(\mathbb{R}^n)]^n$ can be uniquely written as $f = f_S + f_P$, where $f_S$ is divergence-free and $f_P=\nabla\varphi$ for some $\varphi\in H^1(\mathbb{R}^n)$.\footnote{In fact, $\varphi$ is a solution of the Poisson equation $\Delta\varphi=\mathrm{div} f$. For details, we refer the reader to \cite[pp. 81--83]{Sohr}.}
These components are $L^2$-orthogonal. This decomposition yields
$$ \Delta^* u = \mu \Delta u_S + (\lambda + 2\mu)\Delta u_P,$$
which highlights the ellipticity of $\Delta^*$ under the condition \eqref{ellipticity}.

For the solution $u$ of the elastic wave equation \eqref{eq} with $n\ge3$, it was recently shown in \cite{KKS} that 
\begin{equation}\label{LS0}
\sup_{R>0} \frac{1}{R} \int_{|x|<R} \int_{-\infty}^\infty
|u|^2 dtdx \lesssim \|f\|^2_{L^2}+\|g\|_{\dot H^{-1}}^2
\end{equation}
which can be interpreted as a weighted $L^2$ estimate.
Indeed, if $\rho$  is any function such that $\sum_{j\in\mathbb{Z}}\|\rho\|_{L^\infty(|x|\sim2^j)}^2<\infty$,
one has
$$\|\rho|x|^{-1/2}h\|_{L_{x,t}^2}\lesssim\sup_{R>0} \frac{1}{R} \int_{|x|<R} \int_{-\infty}^\infty|h|^2 dtdx,$$
and then \eqref{LS0} implies a weighted estimate,
\begin{equation*}
\|u\|_{L_{x,t}^2(\rho^2|x|^{-1})} \lesssim \|f\|_{L^{2}}^2
+\|g\|_{\dot H^{-1}}^2.
\end{equation*}
A typical example of such $\rho$ is 
$$\rho=|\log|x||^{-1/2-\varepsilon},\quad \varepsilon>0.$$
Note that this case does not include weights of the form $|x|^{-\alpha}$ for any $\alpha>0$.

The aim of this paper is to establish a weighted $L^2$ estimate for such power-type weights. 
This type of estimate was first introduced by Morawetz \cite{M} and is often referred to as a Morawetz estimate.
Our first result is the following.

\begin{thm}\label{thm}
Let $n\geq2$ and let $u$ be the solution to \eqref{eq}. 
Then we have
\begin{equation}\label{mor}
\|u\|_{L_{x,t}^2(|x|^{-\alpha})} \lesssim \|f\|_{\dot H^{s}}+\|g\|_{\dot H^{s-1}}
\end{equation}
if 
\begin{equation*}
0<s<\frac{n-1}2\quad \text{and}\quad \alpha=1+2s.    
\end{equation*}
\end{thm}

The region of $(\alpha, s)$ in the theorem corresponds to the open segments $AB$, $AE$, and $AH$ for $n=2$, $n=3$, and $n \ge 4$, respectively (see Figure~\ref{fig1}).
All these segments lie on the same line through points $A$ and $B$.
As $\alpha$ decreases, so does $s$.
This indicates that the weaker the singularity of the weight, 
the less regularity is required of the initial data in order for the estimate \eqref{mor} to hold.

When a space-time weight such as $|(x,t)|^{-\alpha}$ is considered, the singularity at $x = 0$ is effectively mitigated in regions where $t$ stays away from zero.
This suggests that space-time weights may allow for a further relaxation of the regularity assumptions on the initial data,
as the singularity is no longer concentrated purely in space.
This expectation is precisely confirmed by the following result.

\begin{thm}\label{thm2}
Let $n\geq2$ and let $u$ be the solution to \eqref{eq}. 
Then we have
\begin{equation}\label{mor2}
\|u\|_{L^2_{x,t}(|(x,t)|^{-\alpha})} \lesssim \|f\|_{\dot H^{s}}+\|g\|_{\dot H^{s-1}}
\end{equation}
if 
\begin{equation*}
\frac12<s<\frac{n+1}{4}\quad\text{and}\quad
1+2s<\alpha<4s.
\end{equation*}
\end{thm}

\begin{rem}\label{rem}
Compared to the open segment on which \eqref{mor} holds,
the region of $(\alpha, s)$ where \eqref{mor2} is valid corresponds to the triangles $BCD$, 
$BEF$, and $BGI$ for $n = 2$, $n = 3$, and $n \ge 4$, respectively (see Figure~\ref{fig1}).
The estimate \eqref{mor2} permits a larger range of $\alpha$ than the estimate \eqref{mor}, indicating that it remains valid even under stronger singularities.
Furthermore, for $n \ge 3$, the region where \eqref{mor2} holds lies below the line through $A$ and $B$ that contains the region where \eqref{mor} is valid.
This shows that the estimate \eqref{mor2} requires less regularity on the initial data than the estimate \eqref{mor}.
\end{rem}

\begin{figure}\label{fig1}
\centering
\includegraphics[width=1.0\textwidth]{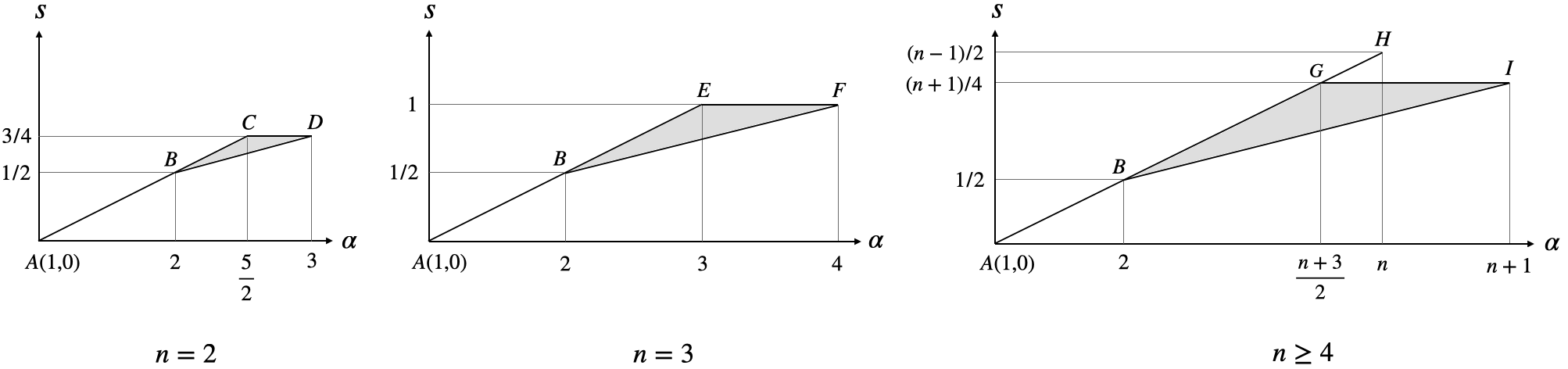}
\caption{The region of $(\alpha,s)$ for \eqref{mor} and \eqref{mor2}.}
\end{figure}

While the estimate \eqref{LS0} in \cite{KKS} relies on the classical Helmholtz decomposition of the initial data, our approach avoids the explicit use of such a decomposition in physical space.
Instead, we decompose the elastic equation in the Fourier domain via orthogonal projections,
which leads to a reduction to classical wave equations.
This enables us to reduce the estimate \eqref{mor} to its counterpart for the classical wave equation, which we then establish.

The proof of \eqref{mor2}, however, is considerably more delicate.  
We employ Littlewood–Paley theory with Muckenhoupt $A_2$ weights to analyze frequency-localized versions of the estimate.  
These localized bounds are derived using a $TT^*$ argument, along with bilinear interpolation applied to the associated bilinear forms.  
In order to obtain the widest possible range of admissible indices $\alpha,s$ for the estimate, we introduce further localizations in both space and time.

It is worth emphasizing that the spacetime weighted estimate proved in this paper fundamentally exploits the dispersive properties of the wave equation. In particular, the estimate on the kernel used in the proof of \eqref{wL2} relies on the fact that the Hessian of the phase of the flat wave propagator is nondegenerate in $n-1$ directions, reflecting the curvature of the characteristic variety. This mechanism yields improved integrability of solutions (as manifested, for example, in Strichartz-type estimates). The same is true for the weighted estimates.
By contrast, for transport equations, even under a uniform transversality assumption between the characteristic flow and a given hypersurface, such dispersive gain is generally difficult to expect. Note that, in the one-dimensional wave equation, the curvature of the characteristic variety disappears and the propagation of solutions essentially reduces to a transport structure. As a result, the dispersive properties employed in this paper do not operate in the same way in that setting. While certain Morawetz-type estimates may persist under such circumstances, the spacetime weighted $L^2$ estimate established here is closely related to the presence of this curvature structure. In this sense, the result may be understood as being inherently dispersive and higher-dimensional in nature.

Such features lie within the classical framework of dispersive and Strichartz estimates for wave equations. The oscillatory integral representation of the wave propagator and the role of curvature in the characteristic variety appear in the foundational work of Strichartz \cite{Str}. These structures underlie the subsequent development (\cite{LS, Mo, K}) of Strichartz estimates, as well as the abstract framework of Keel–Tao \cite{KT} based on the $TT^\ast$ argument and bilinear interpolation. These estimates play a fundamental role in the study of nonlinear dispersive equations; see, for instance, Tao’s monograph \cite{T}.

The rest of the paper is organized as follows.
In Section \ref{sec2}, we derive an explicit formula for the solution to the elastic wave equation, which allows us to reduce the problem to estimating the wave propagator $e^{it\sqrt{-\Delta}}$.
Based on this reduction, we prove the theorems in Section \ref{sec3}.
The estimate \eqref{mor} involving spatial weights is derived from the well-known estimate \eqref{wmor} for the wave propagator.
In contrast, the corresponding estimate \eqref{wmor2} for the space-time weighted estimate \eqref{mor2} is obtained by applying the Littlewood–Paley theory with Muckenhoupt $A_2$ weights to the frequency-localized estimates in Lemma \ref{FLE}.
These localized estimates are derived in Section \ref{sec4} by employing the $TT^*$ argument and subsequently applying bilinear interpolation between their bilinear form estimates given in Lemma \ref{local}.
In Section \ref{sec5}, to obtain the bilinear estimates optimally, we carry out further localizations in both space and time and perform a careful analysis of the relevant oscillatory integrals.

Throughout this paper, the letter $C$ denotes a positive constant which may vary from line to line.
We write $A \lesssim B$ to indicate that $A \leq C B$ for some unspecified constant $C > 0$.

\section{Explicit formula for the solution}\label{sec2}
To prove the theorems in later sections, we first derive an explicit formula for the solution to the elastic wave equation:
\begin{equation*}
\begin{cases}
\partial_t^2 u(x,t) - \Delta^* u(x,t) = 0, \\
u(x,0) = f(x),\quad \partial_t u(x,0) = g(x),
\end{cases}
\end{equation*}
where $\Delta^* u = \mu \Delta u + (\lambda + \mu) \nabla \mathrm{div}\, u$ with $\mu > 0$ and $\lambda + 2\mu > 0$, as introduced earlier.
Taking the Fourier transform in $x$ gives the ODE system:
\begin{equation*}
\begin{cases}
\partial_t^2 \hat{u}(\xi,t) + L(\xi) \hat{u}(\xi,t) = 0, \\
\hat{u}(\xi,0) = \hat{f}(\xi), \quad \partial_t \hat{u}(\xi,0) = \hat{g}(\xi),
\end{cases}
\end{equation*}
where the symbol $L(\xi)$ of $-\Delta^*$ is given by:
\[
L(\xi) = \mu |\xi|^2 I_n + (\lambda + \mu)\, \xi \xi^t.
\]
Let us define the orthogonal projections 
\[
P := \frac{\xi \xi^t}{|\xi|^2}, \quad Q := I_n - P = I_n - \frac{\xi \xi^t}{|\xi|^2}
\]
satisfying 
\(P^2=P\), \(Q^2=Q\), \(PQ=QP=0\) and \(P+Q=I_n\).  
Using this we can rewrite \(L(\xi)\) as
\[
\begin{aligned}
L(\xi)
&=\mu|\xi|^2(P+Q)+(\lambda+\mu)|\xi|^2P \\
&=\mu|\xi|^2Q+(\lambda+2\mu)|\xi|^2P.
\end{aligned}
\]
Thus \(L(\xi)\) has the eigenvalue \(\mu|\xi|^2\) on the subspace \(Q\mathbb{C}^n\) 
and \((\lambda+2\mu)|\xi|^2\) on the subspace \(P\mathbb{C}^n\).

We now decompose the Fourier-transformed solution $\widehat{u}(\xi,t)$ into
\[
\hat{u}(\xi,t) = \hat{u}_Q(\xi,t) + \hat{u}_P(\xi,t),
\]
where $\hat{u}_Q = Q \hat{u}$ and $\hat{u}_P = P \hat{u}$. 
Applying the projection operators $Q$ and $P$ to both sides of the equation
\[
\partial_t^2 \hat{u}(\xi,t) + L(\xi) \hat{u}(\xi,t) = 0,
\]
we obtain
\[
Q\partial_t^2(Q\hat{u}+P\hat{u}) + QL(\xi)(Q\hat{u} + P\hat{u}) = 0
\]
and 
\[
P\partial_t^2(Q\hat{u}+P\hat{u}) + PL(\xi)(Q\hat{u} + P\hat{u}) = 0.
\]
Since $PQ = QP = 0$, $P^2 = P$, and $Q^2 = Q$, and because $L(\xi)$ commutes with both $P$ and $Q$, we have
\begin{align*}
\partial_t^2 \hat{u}_Q(\xi,t) + \mu|\xi|^2 \hat{u}_Q(\xi,t) &= 0, \\
\partial_t^2 \hat{u}_P(\xi,t) + (\lambda + 2\mu)|\xi|^2 \hat{u}_P(\xi,t) &= 0.
\end{align*}
These are decoupled scalar wave equations in Fourier space for each projected component. 
The initial data are similarly decomposed:
\[
\hat{f} = Q\hat{f} + P\hat{f} =: \hat{f}_Q + \hat{f}_P, \qquad
\hat{g} = Q\hat{g} + P\hat{g} =: \hat{g}_Q + \hat{g}_P.
\]
Thus, the solutions to the decoupled systems are given by
\begin{align*}
\hat{u}_Q(\xi,t) &= \cos(t\sqrt{\mu}|\xi|)\, \hat{f}_Q + \frac{\sin(t\sqrt{\mu}|\xi|)}{\sqrt{\mu}|\xi|} \, \hat{g}_Q, \\
\hat{u}_P(\xi,t) &= \cos(t\sqrt{\lambda + 2\mu}|\xi|)\, \hat{f}_P + \frac{\sin(t\sqrt{\lambda + 2\mu}|\xi|)}{\sqrt{\lambda + 2\mu}|\xi|} \, \hat{g}_P.
\end{align*}

Taking inverse Fourier transforms, we obtain the physical-space solution
\begin{equation}\label{soldec}
u(x,t) = u_Q(x,t) + u_P(x,t),
\end{equation}
where
\begin{align*}
u_Q(x,t) &= \mathcal{F}^{-1}\left[ \cos(t\sqrt{\mu}|\xi|)\, Q\hat{f} + \frac{\sin(t\sqrt{\mu}|\xi|)}{\sqrt{\mu}|\xi|} \, Q\hat{g} \right](x), \\
u_P(x,t) &= \mathcal{F}^{-1}\left[ \cos(t\sqrt{\lambda + 2\mu}|\xi|)\, P\hat{f} + \frac{\sin(t\sqrt{\lambda + 2\mu}|\xi|)}{\sqrt{\lambda + 2\mu}|\xi|} \, P\hat{g} \right](x).
\end{align*}
Each component $(u_Q)_j$ of $u_Q$ solves the classical wave equation with wave speed $\sqrt{\mu}$ and initial data given by the $j$-th components of $\mathcal{F}^{-1}(Q\hat{f})$ and $\mathcal{F}^{-1}(Q\hat{g})$.
The same holds for $u_P$, with wave speed $\sqrt{\lambda + 2\mu}$ and initial data $\mathcal{F}^{-1}(P\hat{f})$ and $\mathcal{F}^{-1}(P\hat{g})$.

\section{Estimates for the wave propagator}\label{sec3} 
In this section, we prove Theorems \ref{thm} and \ref{thm2}.
By the Plancherel theorem and the fact that $P$ and $Q$ are orthogonal projections, we observe that
$$
\|\mathcal{F}^{-1}(Q\hat{f})\|_{\dot H^s}+\|\mathcal{F}^{-1}(P\hat{f})\|_{\dot H^s}
\sim \|Q\hat{f}\|_{\dot H^s}+\|P\hat{f}\|_{\dot H^s}
\sim\|\hat{f}\|_{\dot H^s}
\sim\|f\|_{\dot H^s}
$$
and similarly, 
$$
\|\mathcal{F}^{-1}(Q\hat{g})\|_{\dot H^{s-1}}+\|\mathcal{F}^{-1}(P\hat{g})\|_{\dot H^{s-1}}
\sim\|g\|_{\dot H^{s-1}}.
$$
Therefore, using the decomposition \eqref{soldec}, the estimate \eqref{mor} follows from applying the estimate \eqref{wmor} to each component of $\mathcal{F}^{-1}(Q\hat{f})$ and $\mathcal{F}^{-1}(Q\hat{g})$ 
individually. 
Similarly, the estimate \eqref{mor2} follows from applying \eqref{wmor2} to the same components.  
This completes the proof of the theorems.

\begin{prop}\label{prop}
Let $n\ge2$. If $0<s<(n-1)/2$ and $\alpha=1+2s$, then
\begin{equation}\label{wmor}
\big\|e^{it\sqrt{-\Delta}}f\big\|_{L_{x,t}^2(|x|^{-\alpha})} \lesssim \|f\|_{\dot H^{s}}.
\end{equation} 
If 
\begin{equation}\label{con3}
\frac12<s<\frac{n+1}{4}\quad\text{and}\quad 1+2s<\alpha<4s,
\end{equation}
then we have
\begin{equation}\label{wmor2}
\big\|e^{it\sqrt{-\Delta}}f\big\|_{L^2_{x,t}(|(x,t)|^{-\alpha})} \lesssim \|f\|_{\dot H^{s}}.
\end{equation}
\end{prop}

The estimate \eqref{wmor} involving spatial weights is already well known; see, for instance, \cite{HMSSZ}. In contrast, the space-time weighted estimate \eqref{wmor2} appears to be new. The remainder of this paper is devoted to establishing this result.

\begin{proof}[Proof of \eqref{wmor2}]
The proof is based on Littlewood-Paley theory in weighted $L^2$ spaces.

Let us first recall that a \textit{weight}\footnote{A locally integrable function that may vanish or diverge only on a set of Lebesgue measure zero.} 
$w : \mathbb{R}^{n} \to [0, \infty]$ belongs to the Muckenhoupt $A_2(\mathbb{R}^n)$ class if there exists a constant $C_{A_2}$ such that
\[
\sup_{Q\, \textnormal{cubes in } \mathbb{R}^{n}} \left( \frac{1}{|Q|}\int_Q w(x)\,dx \right)\left( \frac{1}{|Q|}\int_Q w(x)^{-1}\,dx \right) < C_{A_2}.
\]
(See, e.g., \cite{G} for details.)
We now observe that the weight $|(x,t)|^{-\alpha}$ belongs to $A_2(\mathbb{R}^{n+1})$ for $-(n+1) < \alpha < n+1$. That is, there exists a constant $C_{A_2}$ such that
\begin{equation} \label{spw}
\sup_{\widetilde{Q} \, \textnormal{cubes in } \mathbb{R}^{n+1}} \left( \frac{1}{|\widetilde{Q}|} \int_{\widetilde{Q}} |(x,t)|^{-\alpha} \,dxdt \right)\left( \frac{1}{|\widetilde{Q}|} \int_{\widetilde{Q}} |(x,t)|^{\alpha} \,dxdt \right) < C_{A_2}.
\end{equation}

Now, let $Q \subset \mathbb{R}^n$ be a cube and $I \subset \mathbb{R}$ be an interval. Applying \eqref{spw} to the product set $\widetilde{Q} = Q \times I$, we obtain
\[
\sup_{Q \subset \mathbb{R}^n,\, I \subset \mathbb{R}} \left( \frac{1}{|I|} \int_I \frac{1}{|Q|} \int_Q |(x,t)|^{-\alpha} \,dxdt \right)\left( \frac{1}{|I|} \int_I \frac{1}{|Q|} \int_Q |(x,t)|^{\alpha} \,dxdt \right) < C_{A_2}.
\]
Since $|(x,t)|^{\pm\alpha} \in L^1_{\mathrm{loc}}(\mathbb{R}^n)$ for $-(n+1) < \alpha < n+1$, we may apply Lebesgue’s differentiation theorem to pass to the limit as $|I| \to 0$. This shows that for almost every fixed $t \in \mathbb{R}$, the function $x \mapsto |(x,t)|^{-\alpha}$ lies in $A_2(\mathbb{R}^n)$, with a uniform bound on the $A_2$ constant.

We are now in a position to apply the Littlewood–Paley theory on weighted $L^2$ spaces with $A_2$ weights (see Theorem~1 in \cite{Ku}) to obtain
\begin{align}\label{LP}
\nonumber\big\| e^{it\sqrt{-\Delta}}\big(\sum_{j\in\mathbb{Z}}P_jf&\big)\big\|_{L^2_{x,t}(|(x,t)|^{-\alpha})}^2
=\int_{\mathbb{R}} \big\|e^{it\sqrt{-\Delta}}\big(\sum_{j\in\mathbb{Z}}P_jf\big)\big\|_{L^2_{x}(|(x,t)|^{-\alpha})}^2 dt \nonumber\\
&\lesssim \int_{\mathbb{R}} \bigg\| \Big(\sum_{k\in\mathbb{Z}} \Big|P_k e^{it\sqrt{-\Delta}}\big(\sum_{j\in\mathbb{Z}}P_jf\big)\Big|^2\Big)^{1/2}\bigg\|_{L^2_{x}(|(x,t)|^{-\alpha})}^2 dt \nonumber\\
&= \int_{\mathbb{R}} \bigg\| \Big(\sum_{k\in\mathbb{Z}} \Big|e^{it\sqrt{-\Delta}}P_k\big(\sum_{|j-k|\leq1}P_jf\big)\Big|^2\Big)^{1/2}\bigg\|_{L^2_{x}(|(x,t)|^{-\alpha})}^2 dt \nonumber\\
&\lesssim \sum_{k\in\mathbb{Z}} \big\| e^{it\sqrt{-\Delta}}P_k\big(\sum_{|j-k|\leq1}P_jf\big)\big\|_{L^2_{x,t}(|(x,t)|^{-\alpha})}^2.
\end{align}
Here we used the fact that $P_k P_j f = 0$ whenever $|j-k|\ge 2$.  
The operators $P_k$, $k \in \mathbb{Z}$, denote the Littlewood–Paley projections defined by
\[
\widehat{P_k f}(\xi) = \phi(2^{-k}|\xi|)\hat{f}(\xi),
\]
where $\phi:\mathbb{R}\to[0,1]$ is a smooth cut-off function supported in $(1/2,2)$ and satisfying 
\[
\sum_{k\in\mathbb{Z}}\phi(2^k t)=1,\quad t>0.
\]
Using \eqref{LP}, we therefore deduce
\begin{align}\label{over}
\big\|e^{it\sqrt{-\Delta}}f\big\|_{L^2_{x,t}(|(x,t)|^{-\alpha})}
\lesssim \Big(\sum_{k\in\mathbb{Z}} \big\| e^{it\sqrt{-\Delta}}P_k\big(\sum_{|j-k|\leq1}P_jf\big)\big\|_{L^2_{x,t}(|(x,t)|^{-\alpha})}^2\Big)^{1/2}.
\end{align}

To estimate the right-hand side of \eqref{over}, we apply the following frequency-localized estimate, which will be established in the next section.

\begin{lem}\label{FLE}
Let $n \ge 2$. Suppose $1 < p < \frac{n+1}{2}$ and $1 + \frac{n+1}{2p} <\alpha < \frac{n+1}{p}$. Then for all $k \in \mathbb{Z}$, we have
\begin{equation}\label{LE2}
\big\|e^{it\sqrt{-\Delta}}P_k f\big\|_{L^2_{x,t}(|(x,t)|^{-\alpha})} \lesssim 2^{k(\frac{n+1}{4p})} \|f\|_{L^2}.
\end{equation}
\end{lem}

Applying Lemma~\ref{FLE}, we estimate
\begin{align}\label{eq22}
\sum_{k\in\mathbb{Z}}\big\| e^{it\sqrt{-\Delta}}P_k\big(\sum_{|j-k|\in\mathbb{Z}}P_jf\big)\big\|_{L^2_{x,t}(|(x,t)|^{-\alpha})}^2
&\lesssim \sum_{k\in\mathbb{Z}}2^{2k(\frac{n+1}{4p})} \Big\|\sum_{|j-k|\le 1} P_j f\Big\|_{L^2}^2 \nonumber\\
&\lesssim   \sum_{k\geq1}2^{2k(\frac{n+1}{4p})} \|P_k f\|_{L^2}^2.
\end{align}
Combining \eqref{over} and \eqref{eq22}, we obtain
\begin{align*}
\big\|e^{it\sqrt{-\Delta}}f\big\|_{L^2_{x,t}(|(x,t)|^{-\alpha})}
&\lesssim\Big(\sum_{k\in\mathbb{Z}}2^{2k(\frac{n+1}{4p})} \|P_k f\|_{L^2}^2\Big)^{1/2}\\
&\lesssim\|f\|_{\dot H^{\frac{n+1}{4p}}}
\end{align*}
if $1<p<\frac{n+1}{2}$ and $ 1+\frac{n+1}{2p} < \alpha < \frac{n+1}{p}$.
Setting $s = \frac{n+1}{4p}$ and eliminating the intermediate parameter $p$ yields the desired condition \eqref{con3}.
Therefore, we get \eqref{wmor2}.
\end{proof}

\section{Frequency-localized estimates: Proof of Lemma \ref{FLE}}\label{sec4}
In this section, we establish the frequency-localized estimate stated in Lemma~\ref{FLE}.  
Our approach is based on the standard $TT^*$ argument, followed by bilinear interpolation between the associated bilinear form estimates.  
To achieve optimal sharpness in \eqref{LE2}, we further introduce additional localizations in both space and time.

We begin by considering the operator
\[
Tf := e^{it\sqrt{-\Delta}}P_k f,
\]
and its adjoint
\[
T^*F := \int_{-\infty}^{\infty} e^{-is\sqrt{-\Delta}} P_k F(\cdot, s) \, ds.
\]
Then, by the standard $TT^*$ argument, the estimate \eqref{LE2} is equivalent to
\begin{equation}\label{df}
\bigg\|\int_{-\infty}^{\infty}  e^{i(t-s)\sqrt{-\Delta}}P_k^2 F(\cdot,s)\, ds\bigg\|_{L^2_{x,t}(|(x,t)|^{-\alpha})}
 \lesssim 2^{\frac{(n+1)k}{2p}}  \|F\|_{L^2_{x,t}(|(x,t)|^{\alpha})}.
\end{equation}
We divide the time integration $\int_{-\infty}^{\infty}$ into two parts, $\int_{-\infty}^{t}$ and $\int_{t}^{\infty}$, and use duality to reduce the proof of \eqref{df} to the following bilinear estimate:
\begin{align}\label{biest}
\bigg|\bigg\langle\int_{-\infty}^{t}  e^{i(t-s)\sqrt{-\Delta}}P_k^2 &F(\cdot,s)  ds, G(x,t)\bigg\rangle_{L_{x,t}^2}\bigg|\nonumber\\
&\lesssim 2^{\frac{(n+1)k}{2p}} \|F\|_{L^2_{x,t}(|(x,t)|^{\alpha})}\|G\|_{L^2_{x,t}(|(x,t)|^{\alpha})}.
\end{align}

To estimate the left-hand side of \eqref{biest}, we perform a dyadic decomposition in time.  
Let $I_j = [2^{j-1}, 2^j)$ for $j \ge 1$, and $I_0 = [0,1)$. For each $j \ge 0$, define the time difference intervals $t - I_j := (t - 2^j, t - 2^{j-1}]$ (or $(t-1, t]$ when $j=0$). Then we can write:
\begin{equation}\label{jsum}
\left\langle \int_{-\infty}^{t} e^{i(t-s)\sqrt{-\Delta}} P_k^2 F(\cdot, s)\, ds,\, G(x,t) \right\rangle_{L^2_{x,t}}
= \sum_{j \ge 0} B_j(F,G),
\end{equation}
where
\[
B_j(F,G) := \int_{\mathbb{R}} \int_{t - I_j} \left\langle e^{i(t-s)\sqrt{-\Delta}} P_k^2 F(\cdot, s),\, G(x,t) \right\rangle_{L^2_x} ds dt.
\]

Now we deduce the desired bilinear estimate \eqref{biest} by interpolating between the following two estimates for each dyadic piece $B_j(F,G)$. These estimates will be proved in the next section.

\begin{lem}\label{local}
Let $n \ge 2$. Then, for each $j \ge 0$, the bilinear form $B_j(F,G)$ satisfies the following estimates:
\begin{equation}\label{L2}
|B_j(F,G)| \le C\, 2^{j} \| F \|_{L^2} \| G \|_{L^2},
\end{equation}
and
\begin{equation}\label{wL2}
|B_j(F,G)| \le C\, 2^{\frac{(n+1)k}{2}}\, 2^{j(\frac{n+3}{2} - \alpha p)} \|F\|_{L^2 (|(x,t)|^{\alpha p})} \|G\|_{L^2 (|(x,t)|^{\alpha p})},
\end{equation}
provided that $\alpha < (n+1)/p$.
\end{lem}

First, we define the bilinear vector-valued operator
\[
B(F,G) = \big\{ B_j(F,G) \big\}_{j \ge 0}.
\]
Then, the bounds \eqref{L2} and \eqref{wL2} can be rephrased as the operator estimates
\[
B : L^2 \times L^2 \longrightarrow \ell_{\infty}^{s_0} \quad \text{and} \quad  B : L^2 (|(x,t)|^{\alpha p}) \times L^2 (|(x,t)|^{\alpha p}) \longrightarrow \ell_{\infty}^{s_1},
\]
with operator norms bounded by $C$ and $C 2^{\frac{(n+1)k}{2}}$, respectively. Here,
\[
s_0 = -1, \quad s_1 = -(\frac{n+3}{2} - \alpha p).
\]
By applying the bilinear interpolation lemma (see Lemma~\ref{biint} below) with the parameters $\theta = 1/p$, $q = \infty$, and $p_1 = p_2 = 2$, we obtain
\begin{align*}
B : \big(L^2, L^2(|(x,t)|^{\alpha p})\big)_{1/p,2} \times \big(L^2, L^2(|(x,t)|^{\alpha p})\big)_{1/p,2}
\longrightarrow \big(\ell_{\infty}^{s_0}, \ell_{\infty}^{s_1}\big)_{1/p,\infty},
\end{align*}
with operator norm bounded by $C^{1-\theta}\cdot(C 2^{\frac{(n+1)k}{2}})^{\theta}=C 2^{\frac{(n+1)k}{2p}}$.

\begin{lem}[\cite{BL}, Section 3.13, Exercise 5(a)]\label{biint}
Let $A_0, A_1, B_0, B_1, C_0, C_1$ be Banach spaces. Suppose that a bilinear operator $T$ satisfies
\[
T : A_0 \times B_0 \rightarrow C_0 \quad \text{and} \quad T : A_1 \times B_1 \rightarrow C_1.
\]
Then, for any $0 < \theta < 1$, $1 \le q \le \infty$, and $1/q = 1/p_1 + 1/p_2 - 1$, we have
\[
T : (A_0, A_1)_{\theta, p_1} \times (B_0, B_1)_{\theta, p_2} \longrightarrow (C_0, C_1)_{\theta, q}.
\]
\end{lem}

Finally, we apply the following real interpolation space identities (see Theorems 5.4.1 and 5.6.1 in \cite{BL}):

\begin{lem}\label{rint}
Let $0 < \theta <1$. Then we have
\begin{equation*}
(L^{2} (w_0), L^{2} (w_1))_{\theta,\,2} = L^2 (w), \quad \text{with } w = w_{0}^{1-\theta} w_{1}^{\theta},
\end{equation*}
and for $1 \leq q_0, q_1, q \leq \infty$ and $s_0 \ne s_1$,
\begin{equation*}
(\ell_{q_0}^{s_0}, \ell_{q_1}^{s_1})_{\theta, q} = \ell_{q}^{s}, \quad \text{where } s = (1-\theta) s_0 + \theta s_1.
\end{equation*}
\end{lem}

\noindent
Applying Lemma~\ref{rint} with $\theta = 1/p$, we obtain
\[
B : L^2(|(x,t)|^{\alpha}) \times L^2(|(x,t)|^{\alpha}) \longrightarrow \ell_{\infty}^{s_2},
\]
with operator norm bounded by $C 2^{\frac{(n+1)k}{2p}}$, provided that $1 < p < \infty$, $\alpha < \frac{n+1}{p}$, and $(n+1)/2 \ne \alpha p$. Here,
\[
s_2 = (1 - \frac{1}{p}) s_0 + \frac{1}{p} s_1 = -(1 + \frac{n+1}{2p} - \alpha).
\]
This is equivalent to the bound
\begin{equation*}
| B_j(F,G) | \lesssim  2^{\frac{(n+1)k}{2p}} \cdot 2^{j(1 + \frac{n+1}{2p} - \alpha)} \, \| F \|_{L^2 (|(x,t)|^{\alpha})} \| G \|_{L^2 (|(x,t)|^{\alpha})}.
\end{equation*}
Now summing over all $j \ge 0$ in the decomposition \eqref{jsum}, the bilinear form estimate \eqref{biest} follows provided
\[
1 + \frac{n+1}{2p} - \alpha < 0,
\]
which ensures convergence of the geometric series.

In summary, the conditions required to establish \eqref{LE2} are
\[
1 < p < \infty, \quad \alpha < \frac{n+1}{p}, \quad \frac{n+1}2 \neq \alpha p, 
\quad \text{and} \quad 1 + \frac{n+1}{2p} - \alpha < 0.
\]
These reduce to the simpler form
\begin{equation*}
1 + \frac{n+1}{2p} < \alpha < \frac{n+1}{p} 
\quad \text{and} \quad 
1 < p < \frac{n+1}{2}.
\end{equation*}
Therefore, \eqref{LE2} holds under exactly the same conditions as those stated in Lemma~\ref{FLE}.

\section{Bilinear estimates: Proof of Lemma \ref{local}}\label{sec5}
Now we proceed to prove the estimates \eqref{L2} and \eqref{wL2} stated in Lemma~\ref{local}.
While the former is relatively straightforward, the proof of the latter requires additional localization in space, as well as a careful analysis of certain oscillatory integrals under this localization (see Lemma~\ref{osci}).

\subsection*{Proof of \eqref{L2}}
For fixed \( j \ge 0 \), we first decompose \( \mathbb{R} \) into intervals of length \( 2^j \), so that
\begin{equation}\label{decomt}
B_j(F,G) = \sum_{\ell \in \mathbb{Z}} \int_{2^j \ell}^{2^j (\ell+1)} \int_{t - I_j} \Big\langle e^{i(t-s) \sqrt{-\Delta}} P_k^2 F(\cdot, s), G(x,t) \Big\rangle_{L_x^2} \, ds \, dt,
\end{equation}
where \( t - I_j = (t - 2^j, t - 2^{j-1}] \) for \( j \ge 1 \) and \( (t - 1, t] \) for \( j = 0 \).

Applying H\"older's inequality and Plancherel's theorem in the \( x \)-variable, we obtain
\begin{align}\label{piece}
\int_{2^j \ell}^{2^j (\ell+1)} & \int_{t - I_j} \Big| \Big\langle  e^{i(t-s) \sqrt{-\Delta}} P_k^2 F(\cdot, s), G(x,t) \Big\rangle_{L_x^2} \Big| \, ds \, dt \nonumber\\
&\le \int_{2^j \ell}^{2^j (\ell+1)} \int_{t - I_j} \big\| e^{i(t-s) \sqrt{-\Delta}} P_k^2 F(\cdot, s) \big\|_{L_x^2} \| G(\cdot, t) \|_{L_x^2} \, ds \, dt \nonumber\\
&\lesssim \int_{2^j \ell}^{2^j (\ell+1)} \int_{t - I_j} \| F(\cdot, s) \|_{L_x^2} \| G(\cdot, t) \|_{L_x^2} \, ds \, dt \nonumber\\
&\le \int_{2^j \ell}^{2^j (\ell+1)} \int_{2^j (\ell - 1)}^{2^j (\ell + \frac{1}{2})} \| F(\cdot, s) \|_{L_x^2} \| G(\cdot, t) \|_{L_x^2} \, ds \, dt.
\end{align}

Using H\"older’s inequality once again and then applying the Cauchy–Schwarz inequality in \( \ell \), we arrive at
\begin{align}\label{lsum}
\sum_{\ell \in \mathbb{Z}} \int_{2^j \ell}^{2^j (\ell+1)} & \int_{2^j (\ell - 1)}^{2^j (\ell + \frac{1}{2})} \| F(\cdot, s) \|_{L_x^2} \| G(\cdot, t) \|_{L_x^2} \, ds \, dt \nonumber\\
&\le C 2^j \sum_{\ell \in \mathbb{Z}} \Big( \int_{2^j (\ell - 1)}^{2^j (\ell + \frac{1}{2})} \| F(\cdot, s) \|_{L_x^2}^2 \, ds \Big)^{1/2}
\Big( \int_{2^j \ell}^{2^j (\ell+1)} \| G(\cdot, t) \|_{L_x^2}^2 \, dt \Big)^{1/2} \nonumber\\
&\le C 2^j \| F \|_{L^2} \| G \|_{L^2}.
\end{align}

Combining \eqref{decomt}, \eqref{piece}, and \eqref{lsum} gives the desired estimate \eqref{L2}.

\subsection*{Proof of \eqref{wL2}}
For fixed \( j \ge 0 \), we further decompose \( \mathbb{R}^n \) into cubes of side length \( 2^j \) and write
\begin{equation}
F(y,s) = \sum_{\rho \in \mathbb{Z}^n} F_{\rho} (y,s) \quad \text{and} \quad G(x,t) = \sum_{\rho \in \mathbb{Z}^n} G_{\rho} (x,t), \nonumber
\end{equation}
where
\[
F_{\rho} (y,s) = \chi_{2^j\rho +[0, 2^j )^n} (y) F(y,s), \quad
G_{\rho} (x,t) = \chi_{2^j \rho +[0, 2^j )^n} (x) G(x,t).
\]
Inserting this decomposition into \eqref{decomt}, we are reduced to showing
\begin{align}\label{reddu}
\nonumber\sum_{\ell \in\mathbb{Z}}\sum_{\rho_1,\rho_2\in\mathbb{Z}^n} & \int_{2^j\ell}^{2^j(\ell+1)}\int_{t-I_j} \Big| \Big\langle e^{i(t-s)\sqrt{-\Delta}}P_k^2 F_{\rho_1}(\cdot,s),G_{\rho_2}(x,t)\Big\rangle_{L_x^2}\Big|dsdt \notag\\
&\lesssim 2^{\frac{(n+1)k}2}  2^{j(\frac{n+3}{2} - \alpha p)}\|F\|_{L_{y,s}^2(|(y,s)|^{\alpha p})}\|G\|_{L_{x,t}^2({|(x,t)|^{\alpha p}})}
\end{align}
for \( \alpha < (n+1)/p \).

To prove this, we first express
\begin{align}\label{PH}
e^{i(t-s)\sqrt{-\Delta}}P_k^2 F_{\rho_1}(\cdot,s)
&= \int_{ \mathbb{R}^n}
\left( \int_{\mathbb{R}^n} e^{i(x-y)\cdot\xi + i(t-s)|\xi|} \, \phi(2^{-k}|\xi|)^2 \, d\xi \right)
F_{\rho_1}(y,s) \, dy \nonumber \\
&:= \int_{ \mathbb{R}^n} I_k(x - y, t - s) \, F_{\rho_1}(y,s) \, dy,
\end{align}
where we define the oscillatory integral kernel
\[
I_k(x - y, t - s) := \int_{\mathbb{R}^n} e^{i(x - y)\cdot\xi + i(t - s)|\xi|} \, \phi(2^{-k}|\xi|)^2 \, d\xi.
\]

We will now estimate this kernel and proceed with the analysis based on this representation.

\begin{lem}\label{osci}
Let \( x \in 2^j \rho_1 + [0,2^j)^n \), \( y \in 2^j \rho_2 + [0,2^j)^n \), and \( s \in t - I_j \). Then,
\begin{equation}\label{nonSP}
\big| I_k(x - y, t - s) \big| \lesssim 2^{nk} \big( 2^{k+j} |\rho_1 - \rho_2| \big)^{-10n}
\end{equation}
when \( |\rho_1 - \rho_2| \ge 4\sqrt{n} \), and
\begin{equation}\label{SP}
\big| I_k(x - y, t - s) \big| \lesssim 2^{\frac{(n+1)k}{2}} 2^{-\frac{(n-1)}{2}j}
\end{equation}
when \( |\rho_1 - \rho_2| < 4\sqrt{n} \).
\end{lem}

\begin{proof}[Proof of Lemma \ref{osci}]
To prove \eqref{nonSP}, we change variables \( \xi \mapsto 2^{k} \xi \) and define \( \psi(\xi) := (x - y) \cdot \xi + (t - s)|\xi| \), \( \varphi(\xi) := \phi(|\xi|)^2 \). Then,
\[
I_k(x - y, t - s) = 2^{nk} \int_{\mathbb{R}^n} e^{i\psi(2^k \xi)} \varphi(\xi) \, d\xi.
\]
If \( |x - y| \ge 2|t - s| \), repeated integration by parts yields
\begin{equation}\label{er}
\bigg| \int_{\mathbb{R}^n} e^{i\psi(2^k \xi)} \varphi(\xi) \, d\xi \bigg| \lesssim (2^k |x - y|)^{-10n}.
\end{equation}
For \( x \in 2^j \rho_1 + [0, 2^j)^n \), \( y \in 2^j \rho_2 + [0, 2^j)^n \), and \( s \in t - I_j \), 
we also observe that
\( |t - s| \le 2^j \) and
\begin{equation}\label{er2}
|2^{-j}x - 2^{-j}y| \ge \frac{1}{2} |\rho_1 - \rho_2| \quad \text{if} \quad |\rho_1 - \rho_2| > \sqrt{n}.
\end{equation}
Hence,
if \( |\rho_1 - \rho_2| \ge 4\sqrt{n} \), then \( |x - y| \ge 2^{j-1}|\rho_1 - \rho_2| \ge 2\sqrt{n}|t - s| \), so \eqref{nonSP} follows from \eqref{er} and \eqref{er2}.

For the second estimate \eqref{SP}, we apply the stationary phase lemma (Lemma \ref{stationary_phase}) due to Littman \cite{L} (see also \cite{St}, Chap. VIII):

\begin{lem}\label{stationary_phase}
Let \( H\psi \) be the Hessian matrix of \( \psi \), given by \( (\partial^2 \psi / \partial \xi_i \partial \xi_j) \). Suppose \( \varphi \) is a compactly supported smooth function on \( \mathbb{R}^n \), and \( \psi \) is a smooth function such that \( \text{rank}(H\psi) \ge k \) on the support of \( \varphi \). Then,
\[
\left| \int_{\mathbb{R}^n} e^{i(x \cdot \xi + t\psi(\xi))} \varphi(\xi) \, d\xi \right| \le C (1 + |(x, t)|)^{-\frac{k}{2}}.
\]
\end{lem}

\noindent In our case,
\begin{align*}
\int_{\mathbb{R}^n} e^{i\psi(2^k \xi)} \varphi(\xi) \, d\xi
&= \int_{\mathbb{R}^n} e^{i 2^k (x - y) \cdot \xi + i 2^k (t - s)|\xi|} \varphi(\xi) \, d\xi \\
&\le C (1 + 2^k |(x - y, t - s)|)^{-\frac{n - 1}{2}} \\
&\le C 2^{-\frac{n - 1}{2}(k + j)},
\end{align*}
since the Hessian of \( |\xi| \) has rank \( n - 1 \) for \( \xi \in \{ \xi \in \mathbb{R}^n : |\xi| \sim 1 \} \). This proves \eqref{SP}.
\end{proof}

Now we return to the proof of \eqref{reddu} with what we just obtained.
Let us first denote
\begin{align*}
C_{\rho_1, \rho_2}(j)=
\begin{cases}
(2^j|\rho_1 -\rho_2|)^{-10n}, & \text{if } |\rho_1-\rho_2|\geq 4\sqrt{n},\\
2^{-\frac{(n-1)j}{2}}, & \text{if } |\rho_1 - \rho_2| < 4\sqrt{n},
\end{cases}
\end{align*}
and set \( C_{k,\sigma}=2^{\frac{(n+1)k}{2}} \).
Using \eqref{PH}, \eqref{nonSP}, and \eqref{SP}, we then have
\begin{align}\label{AfterPH}
&\int_{2^j\ell}^{2^j(\ell+1)} \int_{t-I_j} \Big| \big\langle e^{i(t-s)\sqrt{-\Delta}}P_k^2 F_{\rho_1}(\cdot,s), G_{\rho_2}(x,t) \big\rangle_{L_x^2} \Big|\, ds dt \notag \\
&\quad \lesssim C_{k,\sigma} C_{\rho_1,\rho_2}(j) \int_{2^j\ell}^{2^j(\ell+1)} \int_{t-I_j} \int_{\mathbb{R}^n} \int_{\mathbb{R}^n} |F_{\rho_1}(y,s) G_{\rho_2}(x,t)|\, dy dx ds dt \notag \\
&\quad \lesssim C_{k,\sigma} C_{\rho_1,\rho_2}(j) \int_{2^j\ell}^{2^j(\ell+1)} \int_{2^j(\ell-1)}^{2^j(\ell+\frac{1}{2})} \|F_{\rho_1}(\cdot,s)\|_{L_y^1} \|G_{\rho_2}(\cdot,t)\|_{L_x^1} ds dt.
\end{align}

Now observe that
\begin{align*}
& \int_{2^j(\ell-1)}^{2^j(\ell+\frac{1}{2})} \|F_{\rho_1}(\cdot,s)\|_{L_y^1}ds\\
&\quad = \int_{2^j(\ell-1)}^{2^j(\ell+\frac{1}{2})} \int_{y\in 2^j\rho_1+[0,2^j)^n} |F_{\rho_1}(y,s)|\cdot |(y,s)|^{\alpha p/2}\cdot|(y,s)|^{-\alpha p/2} dyds\\
& \quad \le \|F_{\rho_1}\chi_{[2^j(\ell-1),2^j(\ell+\frac{1}{2}))}\|_{L_{y,s}^{2}(|(y,s)|^{\alpha p})}\Big(\int_{2^j(\ell-1)}^{2^j(\ell+\frac{1}{2})}\int_{y\in 2^j\rho_1+[0,2^j)^n}|(y,s)|^{-\alpha p}\,dyds\Big)^{\frac{1}{2}}\\
&\quad\lesssim 2^{j(n+1-\alpha p)/2}\|F_{\rho_1}\chi_{[2^j(\ell-1),2^j(\ell+\frac{1}{2}))}\|_{L_{y,s}^{2}(|(y,s)|^{\alpha p})}
\end{align*}
if \( \alpha < (n+1)/p \). Similarly,
\[
\int_{2^j\ell}^{2^j(\ell+1)} \|G_{\rho_2}(\cdot,t)\|_{L_x^1} dt \lesssim 2^{\frac{j(n+1-\alpha p)}{2}} \|G_{\rho_2} \chi_{[2^j\ell,2^j(\ell+1))}\|_{L_{x,t}^2(|(x,t)|^{\alpha p})}.
\]
Applying these bounds to \eqref{AfterPH} and using the Cauchy-Schwarz inequality in $\ell$, 
we get
\begin{align}\label{AfterPH2}
&\sum_{\ell\in\mathbb{Z}} \sum_{\rho_1,\rho_2\in\mathbb{Z}^n} \int_{2^j\ell}^{2^j(\ell+1)} \int_{t-I_j} \Big| \big\langle e^{i(t-s)\sqrt{-\Delta}}P_k^2 F_{\rho_1}(\cdot,s), G_{\rho_2}(x,t) \big\rangle_{L_x^2} \Big|\, ds dt \notag \\
&\quad \lesssim C_{k,\sigma} 2^{j(n+1 - \alpha p)} \sum_{\rho_1,\rho_2\in\mathbb{Z}^n} C_{\rho_1,\rho_2}(j) \|F_{\rho_1}\|_{L_{y,s}^2(|(y,s)|^{\alpha p})} \|G_{\rho_2}\|_{L_{x,t}^2(|(x,t)|^{\alpha p})}.
\end{align}

We now estimate the sum on the right-hand side of \eqref{AfterPH2}.  
When $|\rho_1 - \rho_2| < 4\sqrt{n}$, we apply the Cauchy--Schwarz inequality in $\rho_1$ to obtain
\begin{align*}
&\sum_{\substack{\rho_1,\rho_2 \\ |\rho_1-\rho_2|<4\sqrt{n}}} C_{\rho_1,\rho_2}(j) \|F_{\rho_1}\|_{L_{y,s}^2(|(y,s)|^{\alpha p})} \|G_{\rho_2}\|_{L_{x,t}^2(|(x,t)|^{\alpha p})} \\
&\quad \le 2^{-\frac{(n-1)j}{2}} \bigg( \sum_{\rho_1} \|F_{\rho_1}\|_{L_{y,s}^2(|(y,s)|^{\alpha p})}^2 \bigg)^{\frac{1}{2}} \bigg( \sum_{\rho_1} \bigg( \sum_{\substack{\rho_2 \\ |\rho_1-\rho_2|<4\sqrt{n}}} \|G_{\rho_2}\|_{L_{x,t}^2(|(x,t)|^{\alpha p})} \bigg)^2 \bigg)^{\frac{1}{2}}.
\end{align*}
Since the number of $\rho_2$ such that $|\rho_1 - \rho_2|<4\sqrt{n}$ for fixed $\rho_1$ is finite, the above is bounded by
\[
C 2^{-\frac{(n-1)j}{2}} \|F\|_{L_{y,s}^2(|(y,s)|^{\alpha p})} \|G\|_{L_{x,t}^2(|(x,t)|^{\alpha p})}.
\]
Combining this with \eqref{AfterPH2} yields the desired estimate \eqref{reddu} in this case.

Next, consider the case $|\rho_1 - \rho_2| \ge 4\sqrt{n}$. We split the summation into two parts:
\begin{align*}
\sum_{\substack{\rho_1,\rho_2 \\ |\rho_1-\rho_2 |\ge4\sqrt{n}}} C_{\rho_1,\rho_2}(j) \|F_{\rho_1}\|_{L_{y,s}^2(|(y,s)|^{\alpha p})} \|G_{\rho_2}\|_{L_{x,t}^2(|(x,t)|^{\alpha p})} 
= \sum_{\substack{|\rho_1 - \rho_2| \ge 4\sqrt{n} \\ |\rho_1 |\ge 1}} + \sum_{\substack{|\rho_1 - \rho_2| \ge 4\sqrt{n} \\ |\rho_1|< 1}}.
\end{align*}
For the first sum, we use H\"older's inequality in $\rho_2$, followed by Young's inequality:
\begin{align*}
2^{-10nj} &\Big\|\sum_{|\rho_1|\geq 1}(|\rho_1 - \rho_2|)^{-10n} \|F_{\rho_1}\|_{L_{y,s}^2(|(y,s)|^{\alpha p})}\Big\|_{l^2} \Big(\sum_{\rho_2\in\mathbb{Z}^n} \|G_{\rho_2}\|^2_{L_{x,t}^2(|(x,t)|^{\alpha p})}\Big)^{\frac{1}{2}}\notag\\
&\le 2^{-10nj}\Big(\sum_{|\rho_1|\geq 1} |\rho_1|^{-10n}\Big)\Big(\sum_{\rho_1\in\mathbb{Z}^n}\|F_{\rho_1}\|^2_{L_{y,s}^2(|(y,s)|^{\alpha p})}\Big)^{\frac{1}{2}} \|G\|_{L_{x,t}^2({|(x,t)|^{\alpha p}})}\notag\\
&\le C2^{-10nj}\|F\|_{L_{y,s}^2(|(y,s)|^{\alpha p})}\|G\|_{L_{x,t}^2({|(x,t)|^{\alpha p}})}.
\end{align*}
For the second sum, where $|\rho_1|<1$ and hence $\rho_1=0$, we have
\begin{align*}
&\sum_{\{\rho_2:|\rho_2|\geq 4\sqrt{n}\}}(2^{j}|\rho_2|)^{-10n}\|F_0\|_{L_{y,s}^2(|(y,s)|^{\alpha p})}\|G_{\rho_2}\|_{L_{x,t}^2(|(x,t)|^{\alpha p})}\notag\\
&\qquad\le 2^{-10nj}\|F_0\|_{L_{y,s}^2(|(y,s)|^{\alpha p})} \Big(\sum_{\{\rho_2:|\rho_2|\geq 1\}} |\rho_2|^{-20n}\Big)^{\frac{1}{2}}\Big(\sum_{\rho_2\in\mathbb{Z}^n}\|G_{\rho_2}\|^2_{L_{x,t}^2(|(x,t)|^{\alpha p})}\Big)^{\frac{1}{2}}\notag\\
&\qquad\le C2^{-10nj}\|F\|_{L_{y,s}^2(|(y,s)|^{\alpha p})}\|G\|_{L_{x,t}^2({|(x,t)|^{\alpha p}})}.
\end{align*}
Therefore, combining both cases, we conclude that \eqref{reddu} holds when $|\rho_1 - \rho_2| \ge 4\sqrt{n}$ as well.

\section*{Acknowledgments}
The authors thank the anonymous referees for their careful reading of the manuscript and for their insightful comments, which have helped improve the presentation of the paper.

\section*{Data Availability}
This manuscript does not report data generation or analysis.

\section*{Author Contribution}
All authors contributed equally to the preparation of this manuscript.

\end{document}